\documentclass[10pt,journal,epsfig]{IEEEtran}

\usepackage[dvips]{graphicx}
\usepackage{graphicx}
\usepackage{amssymb}
\usepackage{cite}
\usepackage{subfigure}
\usepackage{amsmath}

\begin{document}

\title{Spectral Compressed Sensing via CANDECOMP/PARAFAC Decomposition of Incomplete Tensors}

\author{Jun Fang, Linxiao Yang, and Hongbin Li,~\IEEEmembership{Senior
Member,~IEEE}
\thanks{Jun Fang, and Linxiao Yang are with the National Key Laboratory
of Science and Technology on Communications, University of
Electronic Science and Technology of China, Chengdu 611731, China,
Email: JunFang@uestc.edu.cn}
\thanks{Hongbin Li is
with the Department of Electrical and Computer Engineering,
Stevens Institute of Technology, Hoboken, NJ 07030, USA, E-mail:
Hongbin.Li@stevens.edu}
\thanks{This work was supported in part by the National Science
Foundation of China under Grant 61172114, and the National Science
Foundation under Grant ECCS-1408182. }}

\maketitle

\begin{abstract}
We consider the line spectral estimation problem which aims to
recover a mixture of complex sinusoids from a small number of
randomly observed time domain samples. Compressed sensing methods
formulates line spectral estimation as a sparse signal recovery
problem by discretizing the continuous frequency parameter space
into a finite set of grid points. Discretization, however,
inevitably incurs errors and leads to deteriorated estimation
performance. In this paper, we propose a new method which
leverages recent advances in tensor decomposition. Specifically,
we organize the observed data into a structured tensor and cast
line spectral estimation as a CANDECOMP/PARAFAC (CP) decomposition
problem with missing entries. The uniqueness of the CP
decomposition allows the frequency components to be super-resolved
with infinite precision. Simulation results show that the proposed
method provides a competitive estimate accuracy compared with
existing state-of-the-art algorithms.
\end{abstract}





\section{Introduction}
The problem of recovering the frequency components of a mixture of
complex sinusoids from a finite number of time samples arises in a
variety of applications, such as radar, sonar, array signal
processing and seismology. Such a problem has been extensively
investigated over the past decades and a number of classical
methods such as the MUSIC \cite{Schmidt86}, ESPRIT
\cite{RoyKailath89}, matrix-pencil \cite{HuaSarkar90}, and many
others were proposed in as early as 1980s. In these studies, the
shift invariance property of harmonic structures, i.e. the
subspace of a consecutive segment of time domain samples remains
unaltered irrespective of the starting point of the segment, was
usually exploited for algorithm development. These methods,
however, require that the sampling rate satisfies the
Nyquist-Shannon sampling theorem.





Another line of work that has attracted much attention recently is
to formulate line spectral estimation as a sparse signal recovery
(i.e. compressed sensing) problem. By exploiting the sparsity in
the frequency domain, compressed sensing techniques allow the
frequency components to be recovered from only a small, random
subset of uniformly spaced samples. The sampling rate can thus be
significantly reduced. Also, unlike classical methods
\cite{Schmidt86,RoyKailath89,HuaSarkar90} which assume the
knowledge of the number of frequency components, compressed
sensing methods are able to determine the model order in an
automatic manner. Nevertheless, to apply the compressed sensing
technique to the line spectral estimation problem, one has to
discretize the continuous parameter space into a finite set of
grid points and assumes that the true parameters lie on the
discretized grid. Grid mismatch arises when the true parameters
are inconsistent with the discretized grid, in which case
compressed sensing methods may incur a considerable performance
degradation. To address this issue, a class of off-grid (or
super-resolution) compressed sensing approaches were proposed,
e.g.
\cite{HuShi12,YangXie13,CandesGranda14,TangBhaskar13,BhaskarTang13}.
Specifically, in \cite{CandesGranda14,TangBhaskar13}, an atomic
norm-minimization approach was proposed to handle the infinite
dictionary with continuous atoms. It was shown that given that the
frequency components are sufficiently separated, the frequency
components of a mixture of complex sinusoids can be super-resolved
with infinite precision from only coarse-scale measurements. Also,
in \cite{ChenChi14,CaiLiu15}, by arranging the observed samples
into a low-rank Hankel matrix, a structured matrix completion
method was developed to recover the real-valued frequency
parameters.




In this paper, we propose a new method which organizes the
observed data into a structured tensor and cast line spectral
estimation as a CANDECOMP/PARAFAC (CP) decomposition problem with
missing entries. Due to the uniqueness of the CP decomposition,
the frequency components of a mixture of complex sinusoids is
guaranteed to be super-resolved from only a small number of
nonuniform samples. Simulation results show that the proposed
method provides competitive estimation performance compared with
existing state-of-the-art algorithms.

\section{Preliminaries}
We first provide a brief review on tensor and the CP
decomposition. A tensor is the generalization of a matrix to
higher order dimensions, also known as ways or modes. Vectors and
matrices can be viewed as special cases of tensors with one and
two modes, respectively.

Let $\boldsymbol{\mathcal{X}}\in\mathbb{R}^{I_1\times
I_2\times\cdots\times I_N}$ denote an $N$th order tensor with its
$(i_1,\ldots,i_N)$th entry denoted by $\mathcal{X}_{i_1\cdots
i_N}$. Here the order $N$ of a tensor is the number of dimensions.
Fibers are the higher-order analogue of matrix rows and columns.
The mode-$n$ fibers of $\boldsymbol{\mathcal{X}}$ are
$I_n$-dimensional vectors obtained by fixing every index but
$i_n$. Unfolding or matricization is an operation that turns a
tensor to a matrix. Specifically, the mode-$n$ unfolding of a
tensor $\boldsymbol{\mathcal{X}}$, denoted as
$\boldsymbol{X}_{(n)}$, arranges the mode-$n$ fibers to be the
columns of the resulting matrix. The $n$-mode product of
$\boldsymbol{\mathcal{X}}$ with a matrix
$\boldsymbol{A}\in\mathbb{R}^{J\times I_n}$ is denoted by
$\boldsymbol{\mathcal{X}}\times_n\boldsymbol{A}$ and is of size
$I_1\cdots\times I_{n-1}\times J\times I_{n+1}\times\cdots\times
I_N$, with each mode-$n$ fiber multiplied by the matrix
$\boldsymbol{A}$, i.e.
\begin{align}
\boldsymbol{\mathcal{Y}}=\boldsymbol{\mathcal{X}}\times_n\boldsymbol{A}\Leftrightarrow
\boldsymbol{Y}_{(n)}=\boldsymbol{A}\boldsymbol{X}_{(n)}
\end{align}

The CP decomposition decomposes a tensor into a sum of component
rank-one component tensors, i.e.
\begin{align}
\boldsymbol{\mathcal{X}}=
\sum\limits_{r=1}^{R}\lambda_r\boldsymbol{a}_r^{(1)}\circ\boldsymbol{a}_r^{(2)}\circ\cdots\circ\boldsymbol{a}_r^{(N)}
\end{align}
where $\boldsymbol{a}_r^{(n)}\in\mathbb{R}^{I_n}$, `$\circ$'
denotes the vector outer product, the minimum achievable $R$ is
referred to as the rank of the tensor, and
$\boldsymbol{A}^{(n)}\triangleq
[\boldsymbol{a}_{1}^{(n)}\phantom{0}\ldots\phantom{0}\boldsymbol{a}_{R_n}^{(n)}]\in\mathbb{R}^{I_n\times
R_n}$ denotes the factor matrix along the $n$-th mode.
Elementwise, we have
\begin{align}
\mathcal{X}_{i_1 i_2\cdots i_N}=\sum\limits_{r=1}^{R}\lambda_r
a_{i_1 r}^{(1)}a_{i_2 r}^{(2)}\cdots a_{i_N r}^{(N)}
\end{align}



\section{Tensor Formulation}
Consider the line spectral estimation problem where the signal
$x_n$ can be represented as a summation of a number of complex
sinusoids, i.e.
\begin{align}
x_n=\sum\limits_{k=1}^{K}a_k e^{-j\omega_k (n-1)},\quad
n=1,2,\dots,N
\end{align}
where $\omega_k\in[0,2\pi)$ and $a_k$ denote the frequency and
complex amplitude of the $k$-th component, respectively. Let
$\boldsymbol{y}\triangleq [y_1 \phantom{0}y_2
\phantom{0}\dots\phantom{0} y_M]^T$ denote the noise-corrupted
observations randomly chosen from the original set $\{x_n\}$. Our
objective is to estimate $\{a_k\}$ and $\{\omega_k\}$ from the
observed data $\boldsymbol{y}$. To this objective, we organize the
observed samples $\{y_m\}_{m=1}^M$ as an incomplete structured
third-order tensor. The unknown parameters $\{\omega_k\}$ along
with the missing entries can be estimated via the CP decomposition
of this incomplete tensor.

To better illustrate our method, we first show how to construct
the tensor using the original samples $\{x_n\}$ instead of the
observed samples $\{y_n\}$. In order to obtain a tensor which
admits a CP decomposition, we arrange samples $\{x_n\}$ to form a
third-order tensor
$\boldsymbol{\mathcal{X}}\in\mathbb{C}^{(N-L-P+2)\times L\times
P}$, with each slice along the third-mode of the tensor being an
$(N-L-P+2)\times L$ matrix, i.e.
\begin{align}
\boldsymbol{\mathcal{X}}(:,:,i)=[\boldsymbol{x}_{L+i-1}\phantom{0}\boldsymbol{x}_{L+i-2}\phantom{0}
\dots\phantom{0}\boldsymbol{x}_i]
\end{align}
where $L$ and $P$ are parameters whose choice will be discussed
later, and
\begin{align}
\boldsymbol{x}_t\triangleq
[x_t\phantom{0}x_{t+1}\phantom{0}\dots\phantom{0}x_{N-L-P+t+1}]^T
\quad t=1,\dots,L+P-1
\end{align}
By exploiting the inherent structure, each slice
$\boldsymbol{\mathcal{X}}(:,:,i)$ can be expressed as
\begin{align}
\boldsymbol{\mathcal{X}}(:,:,i)=\boldsymbol{A}\boldsymbol{D}_i\boldsymbol{B}^T
\quad i=1,\dots,P
\end{align}
where $\boldsymbol{A}\in\mathbb{C}^{(N-L-P+2)\times K}$,
$\boldsymbol{B}\in\mathbb{C}^{L\times K}$ and
$\boldsymbol{D}_i\in\mathbb{C}^{K\times K}$ are defined
respectively as
\begin{align}
\boldsymbol{A}\triangleq\begin{bmatrix}
1 & \cdots & 1\\
e^{-j\omega_1} & \cdots & e^{-j\omega_K}\\
\vdots&&\vdots\\
e^{-j\omega_1 (N-L-P+1)} & \cdots & e^{-j\omega_K (N-L-P+1)}\\
\end{bmatrix}
\end{align}
\begin{align}
\boldsymbol{B}\triangleq\begin{bmatrix}
e^{-j\omega_1 (L-1)} & \cdots & e^{-j\omega_K (L-1)}\\
e^{-j\omega_1 (L-2)} & \cdots & e^{-j\omega_K (L-2)}\\
\vdots&&\vdots\\
1 & \cdots & 1\\
\end{bmatrix}
\end{align}
and
\begin{align}
\boldsymbol{D}_i \triangleq \text{diag}(a_1 e^{-j\omega_1
(i-1)},a_2 e^{-j\omega_2 (i-1)},\cdots, a_K e^{-j\omega_K (i-1)})
\end{align}
Let $\boldsymbol{a}_k$ and $\boldsymbol{b}_k$ denote the $k$th
column of $\boldsymbol{A}$ and $\boldsymbol{B}$, respectively. The
slice $\boldsymbol{\mathcal{X}}(:,:,i)$ can be rewritten as
\begin{align}
\boldsymbol{\mathcal{X}}(:,:,i)=\sum_{k=1}^{K}a_k e^{-j\omega_k
(i-1)}\boldsymbol{a}_k\boldsymbol{b}_k^T
\end{align}
We see that each slice of $\boldsymbol{\mathcal{X}}$ is a weighted
sum of a common set of rank-one outer products. Hence the tensor
$\boldsymbol{\mathcal{X}}$ admits the following CP decomposition
which decomposes a tensor as a sum of component rank-one tensors,
i.e.
\begin{align}
\boldsymbol{\mathcal{X}}=\sum\limits_{k=1}^{K}\boldsymbol{a}_k\circ\boldsymbol{b}_k\circ\boldsymbol{c}_k
\end{align}
where $\boldsymbol{c}_k\triangleq [a_k\phantom{0} a_k
e^{-j\omega_k}\phantom{0}\ldots\phantom{0}a_k
e^{-j\omega_k(M-1)}]^T$. Define $\boldsymbol{C}\triangleq
[\boldsymbol{c}_1\phantom{0}\ldots\phantom{0}\boldsymbol{c}_K]$.
The matrices $\boldsymbol{A}$, $\boldsymbol{B}$ and
$\boldsymbol{C}$ are factor matrices associate with the tensor
$\boldsymbol{\mathcal{X}}$. Since $K$ is usually small, the above
factorization implies that the tensor $\boldsymbol{\mathcal{X}}$
has a low-rank structure.


\section{Algorithm Development}
When only the observations $\{y_m\}$, are available, we can
readily construct an incomplete third-order tensor
$\boldsymbol{\mathcal{Y}}$ by following the way we construct
$\boldsymbol{\mathcal{X}}$. By exploiting the low rank structure,
the missing entries of $\boldsymbol{\mathcal{Y}}$, along with the
factor matrices that contain information about the frequencies,
can be estimated. Specifically, the problem can be cast as
\begin{align}
\min_{\boldsymbol{\mathcal{X}}}&\quad
\text{rank}(\boldsymbol{\mathcal{X}}) \nonumber\\
\text{s.t.}&\quad
\|\boldsymbol{\mathcal{O}}\ast\boldsymbol{\mathcal{Y}}-
\boldsymbol{\mathcal{O}}\ast\boldsymbol{\mathcal{X}})\|_F^2\leq\varepsilon
\label{opt-1}
\end{align}
where $\boldsymbol{\mathcal{O}}$ is a binary tensor of the same
size as $\boldsymbol{\mathcal{X}}$ with $\mathcal{O}_{ijk}=1$ if
$\mathcal{X}_{ijk}$ is observed, and $\mathcal{O}_{ijk}=0$
otherwise. $\varepsilon$ is an error tolerance parameter related
to noise statistics. Note that the CP rank is the minimum number
of rank-one tensor components required to represent the tensor.
Thus the search of a low rank $\boldsymbol{\mathcal{X}}$ can be
converted to optimization of its associated factor matrices. Let
\begin{align}
\boldsymbol{\mathcal{X}}=\sum\limits_{k=1}^{\tilde{K}}\boldsymbol{\tilde{a}}_k\circ
\boldsymbol{\tilde{b}}_k\circ\boldsymbol{\tilde{c}}_k
\end{align}
where $\tilde{K}\gg K$, and
\begin{align}
\boldsymbol{\tilde{A}}\triangleq &
[\boldsymbol{\tilde{a}}_1\phantom{0}\ldots\phantom{0}\boldsymbol{\tilde{a}}_{\tilde{K}}]
\nonumber\\
\boldsymbol{\tilde{B}}\triangleq &
[\boldsymbol{\tilde{b}}_1\phantom{0}\ldots\phantom{0}\boldsymbol{\tilde{b}}_{\tilde{K}}]
\nonumber\\
\boldsymbol{\tilde{C}}\triangleq &
[\boldsymbol{\tilde{c}}_1\phantom{0}\ldots\phantom{0}\boldsymbol{\tilde{c}}_{\tilde{K}}]
\nonumber
\end{align}
The optimization (\ref{opt-1}) can be re-expressed as
\begin{align}
\min_{\boldsymbol{\tilde{A}},\boldsymbol{\tilde{B}},\boldsymbol{\tilde{C}}}&\quad
\|\boldsymbol{z}\|_0 \nonumber\\
\text{s.t.} &\quad
\|\boldsymbol{\mathcal{O}}\ast\boldsymbol{\mathcal{Y}}-
\boldsymbol{\mathcal{O}}\ast\boldsymbol{\mathcal{X}})\|_F^2\leq\varepsilon
\nonumber\\
&\quad
\boldsymbol{\mathcal{X}}=\sum\limits_{k=1}^{\tilde{K}}\boldsymbol{\tilde{a}}_k\circ
\boldsymbol{\tilde{b}}_k\circ\boldsymbol{\tilde{c}}_k
\label{opt-2}
\end{align}
where $\boldsymbol{z}$ is a $\tilde{K}$-dimensional vector with
its $k$th entry given by
\begin{align}
z_{k}\triangleq\|\boldsymbol{\tilde{a}}_k\circ
\boldsymbol{\tilde{b}}_k\circ\boldsymbol{\tilde{c}}_k\|_{F}
\end{align}
We see that $\|\boldsymbol{z}\|_{0}$ equals to the number of
nonzero rank-one tensor components. Therefore minimizing the
$\ell_0$-norm of $\boldsymbol{z}$ is equivalent to minimizing the
rank of the tensor $\boldsymbol{\mathcal{X}}$.

The optimization (\ref{opt-2}) is an NP-hard problem.
Nevertheless, alternative sparsity-promoting functions such as
$\ell_1$-norm can be used to replace $\ell_0$-norm to find a
sparse solution of $\boldsymbol{z}$ more computationally
efficient. Based on the idea of placing sparsity on the rank-one
tensor components, a few CP decomposition algorithms were recently
proposed via either optimization techniques
\cite{MateosGiannakis13} or probabilistic model learning
\cite{RaiWang14,ZhaoZhang15}. We have no intention to develop a
new algorithm in this paper as our objective is to show how to
formulate the line spectral estimation problem as a low rank CP
decomposition problem. Once an estimate of the factor matrices is
obtained, the underlying frequencies can be easily identified
since all three factor matrices are Vandermonde matrices and each
column of the Vandermonde matrix is associated with an individual
frequency parameter.


\section{Uniqueness of CP Decomposition}
Although the factor matrices have a specific structure, we do not
need to impose a specific structure on the estimates of the factor
matrices since the CP decomposition is unique under very mild
conditions. It is well know that essential uniqueness of the CP
decomposition can be guaranteed by the Kruskal's condition
\cite{StegemanSidiropoulos07}. Let $k_{\boldsymbol{X}}$ denote the
k-rank of a matrix $\boldsymbol{X}$, which is defined as the
largest value of $k_{\boldsymbol{X}}$ such that every subset of
$k_{\boldsymbol{X}}$ columns of the matrix $\boldsymbol{X}$ is
linearly independent. Kruskal showed that a CP decomposition
$(\boldsymbol{A},\boldsymbol{B},\boldsymbol{C})$ of a three-order
tensor is essentially unique if
\begin{align}
k_{\boldsymbol{A}}+k_{\boldsymbol{B}}+k_{\boldsymbol{C}}\geq 2R+2
\label{Kruskals-condition}
\end{align}
where $\boldsymbol{A},\boldsymbol{B},\boldsymbol{C}$ are factor
matrices, $R$ denotes the CP rank. More formally, we have the
following theorem.

\newtheorem{theorem}{Theorem}
\begin{theorem} \label{theorem1}
Let $(\boldsymbol{A},\boldsymbol{B},\boldsymbol{C})$ be a CP
solution which decomposes a three-mode tensor
$\boldsymbol{\mathcal{X}}$ into $R$ rank-one arrays. Suppose
Kruskal's condition (\ref{Kruskals-condition}) holds and we have
an alternative CP solution
$(\boldsymbol{\bar{A}},\boldsymbol{\bar{B}},\boldsymbol{\bar{C}})$
also decomposing $\boldsymbol{\mathcal{X}}$ into $R$ rank-one
arrays. Then there holds
$\boldsymbol{\bar{A}}=\boldsymbol{A}\boldsymbol{\Pi}\boldsymbol{\Lambda}_a$,
$\boldsymbol{\bar{B}}=\boldsymbol{B}\boldsymbol{\Pi}\boldsymbol{\Lambda}_b$,
and
$\boldsymbol{\bar{C}}=\boldsymbol{C}\boldsymbol{\Pi}\boldsymbol{\Lambda}_c$,
where $\boldsymbol{\Pi}$ is a unique permutation matrix and
$\boldsymbol{\Lambda}_a$, $\boldsymbol{\Lambda}_b$, and
$\boldsymbol{\Lambda}_c$ are unique diagonal matrices such that
$\boldsymbol{\Lambda}_a\boldsymbol{\Lambda}_b\boldsymbol{\Lambda}_c=\boldsymbol{I}$.
\end{theorem}
\begin{proof}
Please refer to \cite{StegemanSidiropoulos07}.
\end{proof}

We now discuss how to choose $P$ and $L$ such that the Kruskal's
condition can be met. Note that for the line spectral estimation
problem, all three factor matrices $\boldsymbol{A}$,
$\boldsymbol{B}$, and $\boldsymbol{C}$ are Vandermonte matrices.
Hence the k-rank of each factor matrix is equivalent to the
minimum value of the numbers of columns and rows, i.e.
\begin{align}
k_{\boldsymbol{A}}=&\text{min}(N-L-P+2,K) \nonumber\\
k_{\boldsymbol{B}}=&\text{min}(L,K) \nonumber\\
k_{\boldsymbol{C}}=& \text{min}(P,K) \nonumber
\end{align}
In order to satisfy the Kruskal's condition, we can choose one of
the three dimensions, say $P$, equal to 2, and the other two
dimensions $N-L-P+2$ and $L$ no less than $K$. Note that when
$R=1$, the Kruskal's condition (\ref{Kruskals-condition}) cannot
be satisfied whatever $P$ and $L$ we choose. Nevertheless, the
uniqueness of the CP decomposition also holds for this special
case as long as $\boldsymbol{\mathcal{X}}$ does not contain an
identically zero two-dimensional slice along any mode
\cite{SidiropoulosGiannakis00}.

We would like to emphasize that the tensor
$\boldsymbol{\mathcal{X}}$ is assumed fully observed in Theorem
\ref{theorem1} when discussing the uniqueness of the CP
decomposition. It still remains an open problem whether the
uniqueness holds valid if only a subset of the entries of the
third-order tensor are available/observed, which is exactly the
situation we are concerned in this paper. This will be a topic of
our future investigation.

\begin{figure*}[!t]
 \centering
\begin{tabular}{ccc}
\hspace*{-2ex}
\includegraphics[width=6cm]{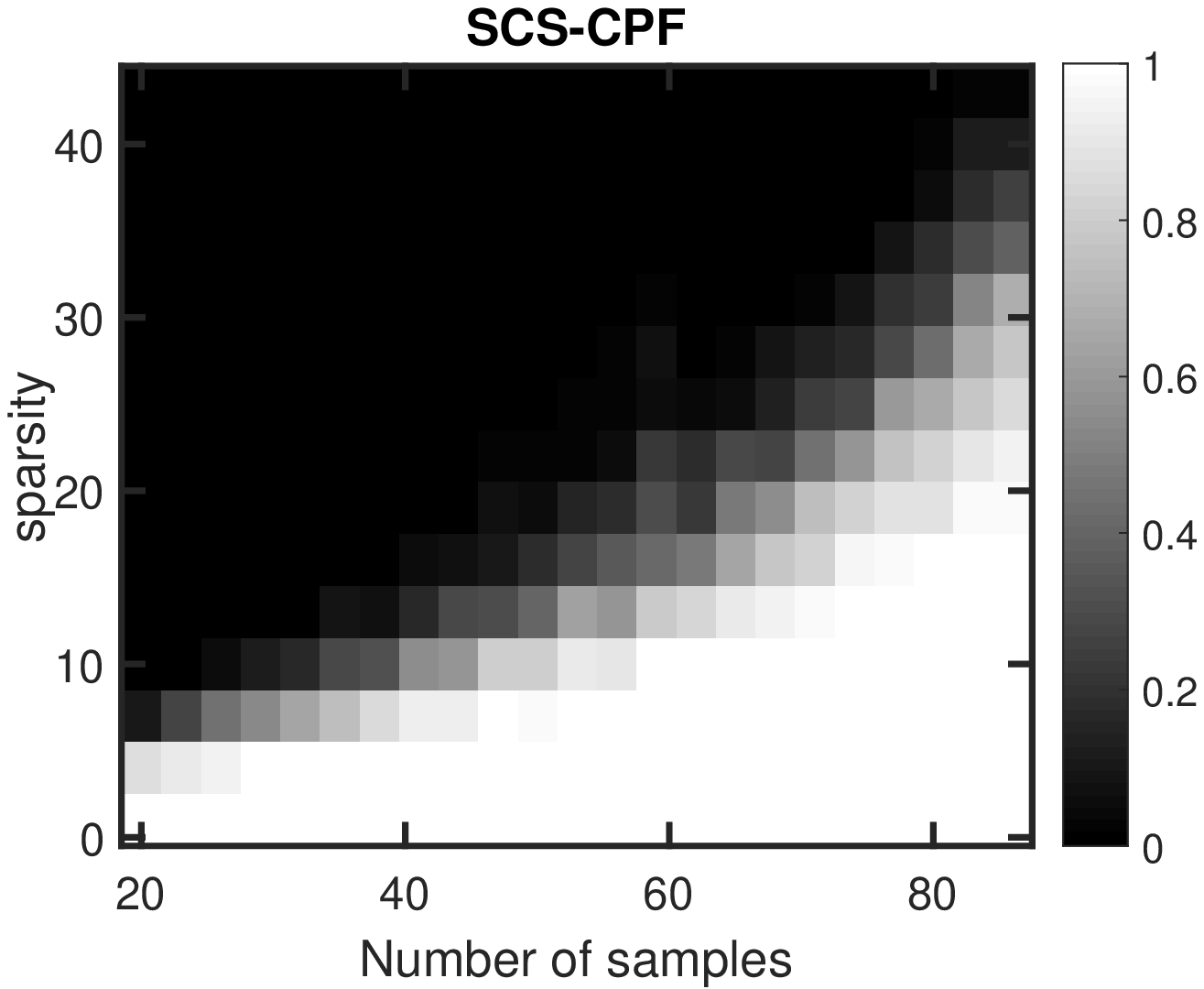} &
\hspace*{-2ex}
\includegraphics[width=6cm]{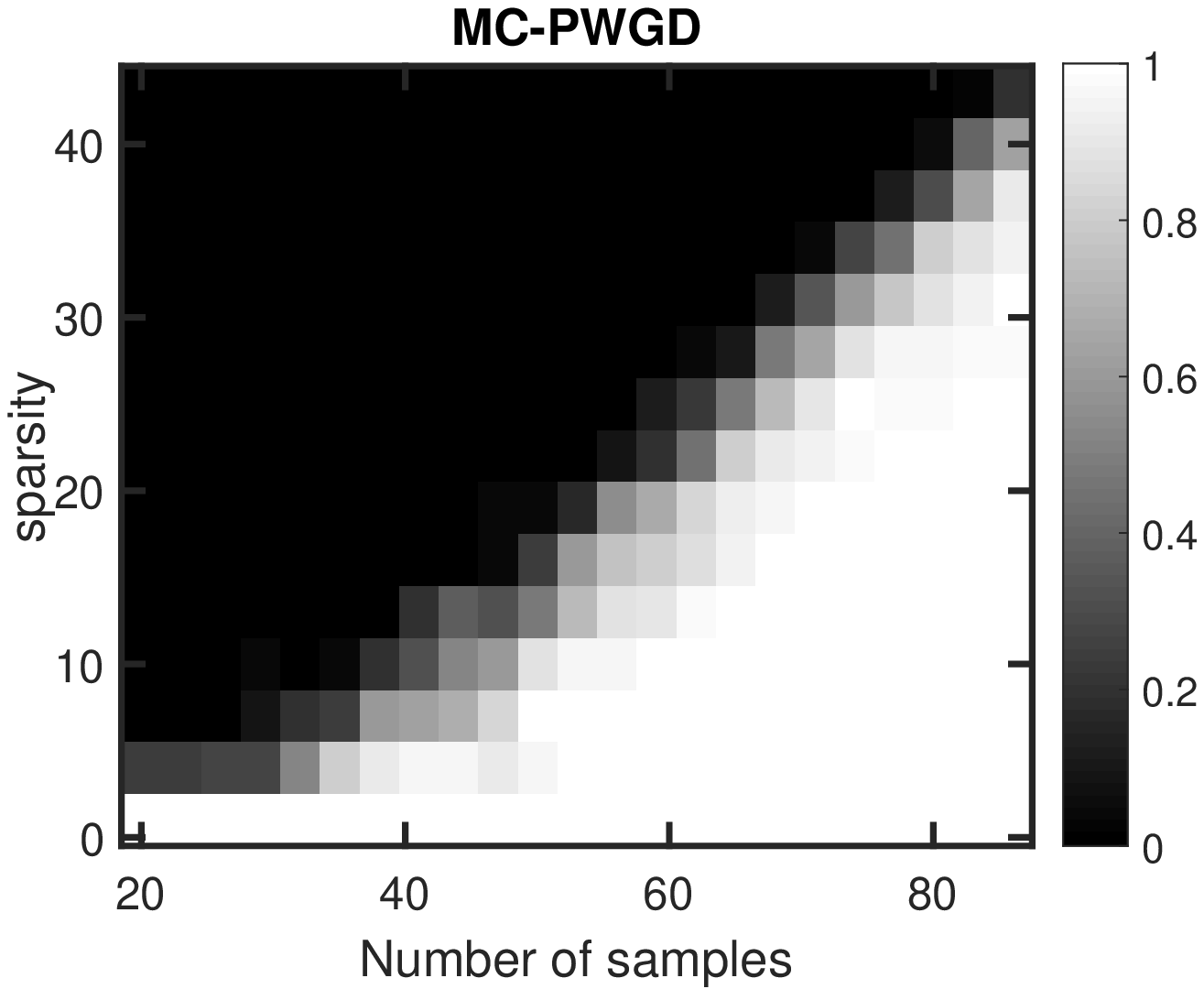} &
\hspace*{-2ex}
\includegraphics[width=6cm]{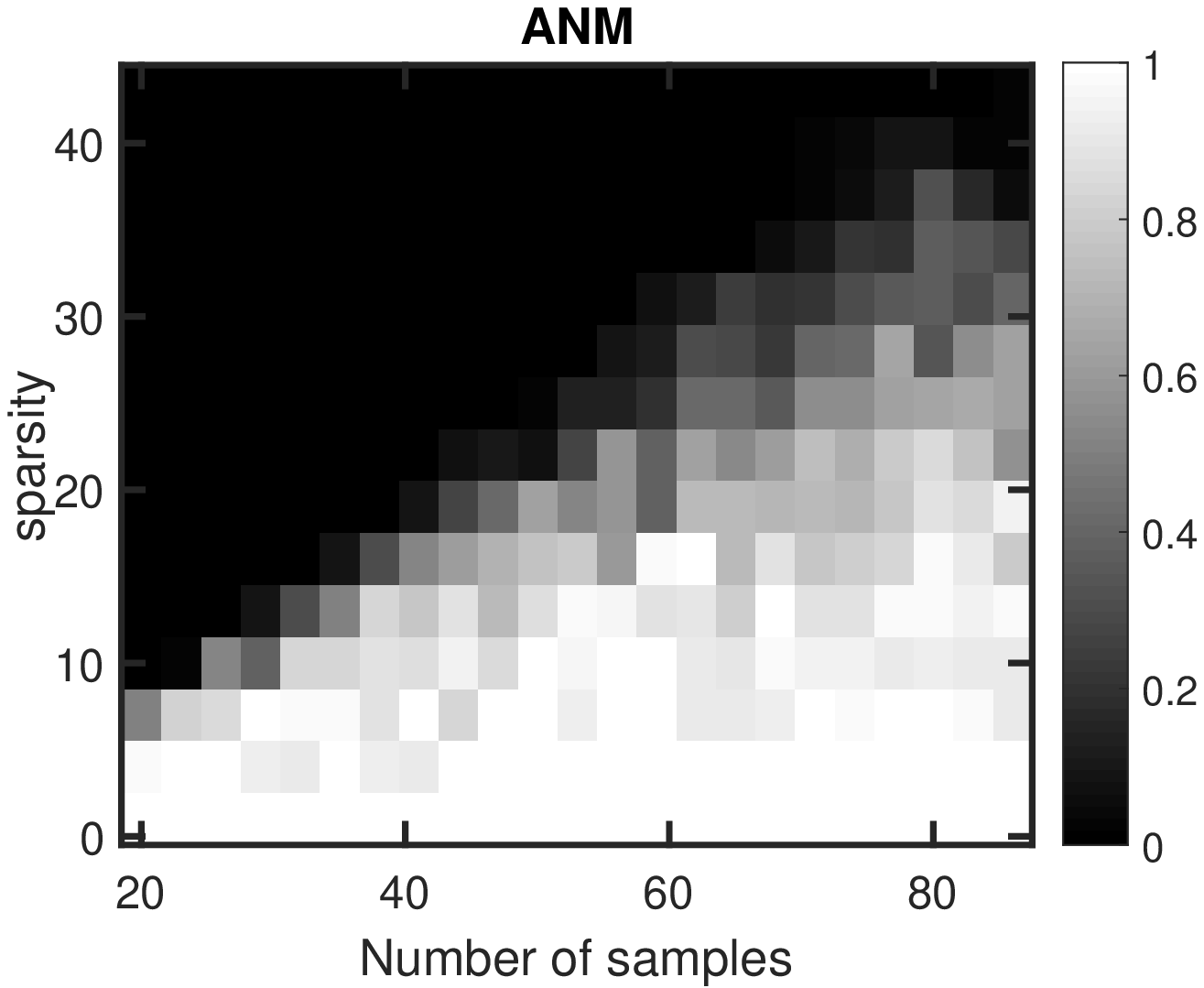}
\end{tabular}
  \caption{Phase transitions of respective algorithms.}
   \label{fig1}
\end{figure*}

\begin{figure*}[!t]
 \centering
\subfigure[RSNRs vs. $M$.]{\includegraphics[width=8cm]{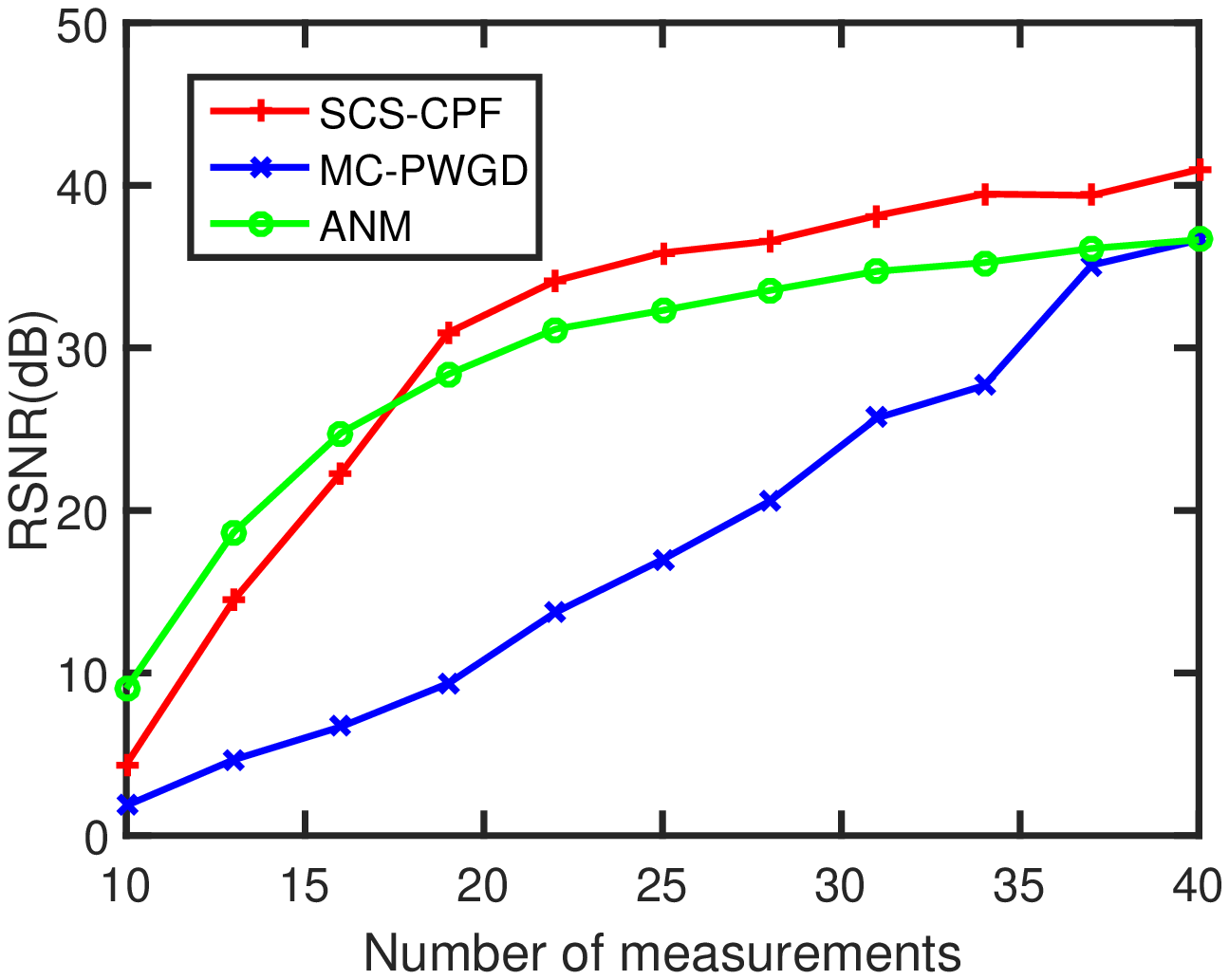}}
 \hfil
\subfigure[RSNR vs. PSNR]{\includegraphics[width=8cm]{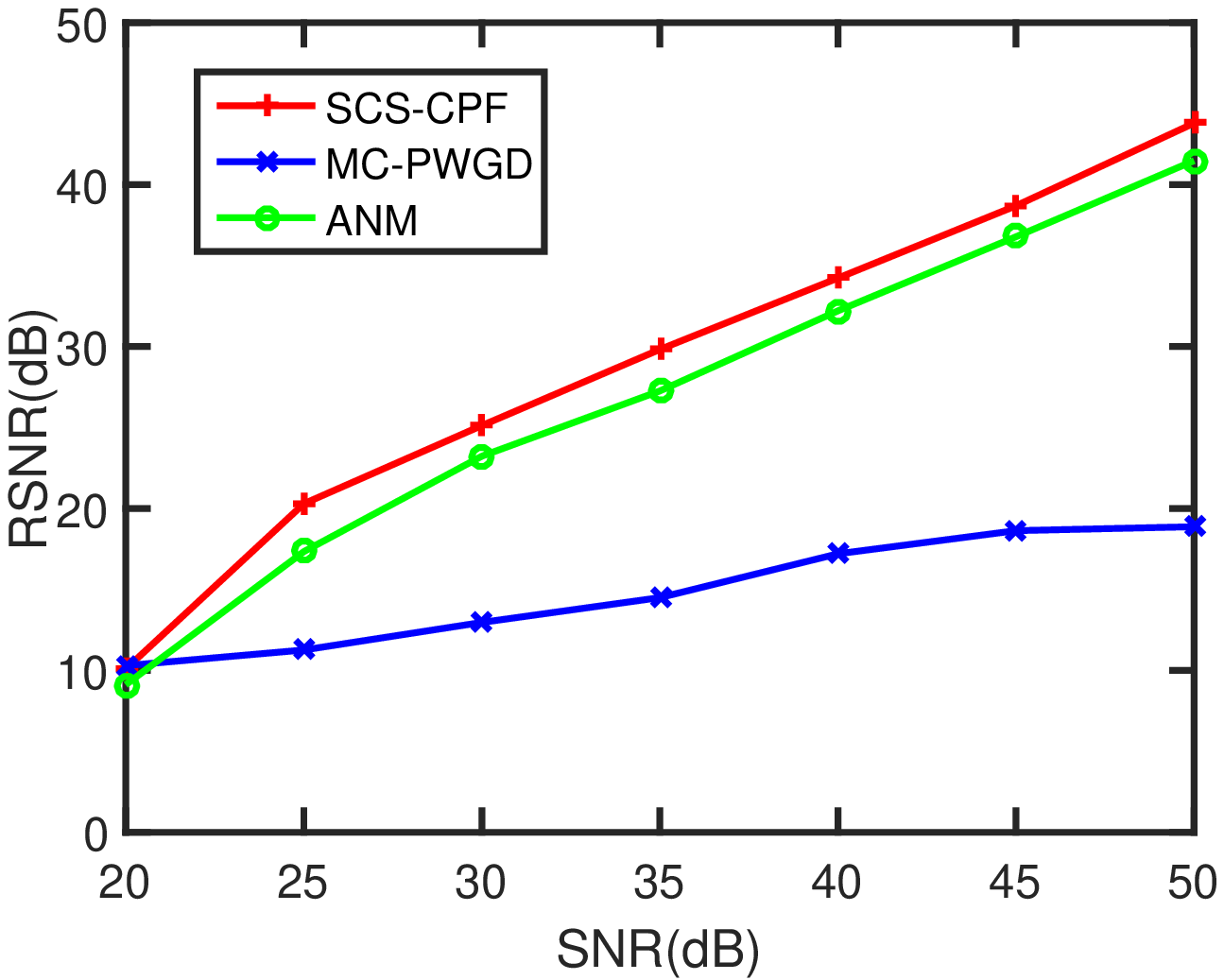}}
  \caption{RSNRs of respective algorithms vs. $M$ and PSNR.}
   \label{fig2}
\end{figure*}

\section{Simulation Results}
We now carry out experiments to illustrate the performance of the
proposed method which is referred to as Spectral Compressed
Sensing via CP Factorization (SCS-CPF). We compare our method with
the Hankel matrix completion method via the projected Wirtinger
gradient descent (MC-PWGD) \cite{CaiLiu15}, and the atomic norm
minimization approach (ANM) \cite{TangBhaskar13,BhaskarTang13}.
For our method, a Bayesian decomposition technique
\cite{ZhaoZhang15} is employed to perform the CP factorization of
the constructed incomplete tensor. The Bayesian algorithm is able
to achieve an automatic determination of the rank of the tensor.
Also, model parameters associated with the Bayesian approach can
be simply chosen and do not require a careful calibration. In our
simulations, we choose $N=127$ as the MC method \cite{CaiLiu15}
requires an odd number of $N$ to form an $(N+1)/2\times (N+1)/2$
low rank Hankel matrix. The frequencies $\{\omega_k\}$ are
uniformly distributed over $[0,2\pi)$ and the amplitudes $\{a_k\}$
are randomly generated according to a normal distribution. The
parameters $P$ and $L$ are chosen to be 2 and 63, respectively,
for our proposed method. Thus $\boldsymbol{\mathcal{X}}$ is of
size $64\times 63\times 2$. Note that for the noisy case, the ANM
method requires the knowledge of the noise variance, which is
assumed perfectly known to the ANM. Also, the MC-PWGD method
requires the knowledge of the number of frequency components,
which is assumed known to the MC-PWGD.

We first consider a noiseless case and plot the phase transition
curve for each algorithm. We vary the sparsity level $K=3:2:43$
and the number of measurements $M=20:3:86$. For each point
$(M,K)$, we conduct 100 independent trials and compute the success
rate. A trial is considered successful if the normalized
reconstruction error is smaller than $10^{-3}$, i.e.
$\|\boldsymbol{x}-\boldsymbol{\hat{x}}\|_2/\|\boldsymbol{x}\|_2<10^{-3}$,
where $\boldsymbol{x}\triangleq
[x_1\phantom{0}x_2\phantom{0}\ldots\phantom{0}x_N]^T$ denotes the
original signal and $\boldsymbol{\hat{x}}$ denotes the estimated
one. In the phase transition plot, the grey value of each point
represents the success rate, with white corresponding to perfect
recovery while black corresponding to complete failure. We can see
from Fig. \ref{fig1} that the proposed SCS-CPF method outperforms
the MC-PWGD method for a small $M$ (e.g. $M\leq 50$), where data
acquisition is more beneficial due to high compression rates. The
ANM method has a sharper transition boundary that the other two
methods. Nevertheless, its transition boundary is highly blurred
and the size of the white area below the transition boundary is
smaller than those of the other two methods, which implies that
the ANM is inferior to the other two methods in terms of perfect
recovery rates.


We now evaluate the recovery performance of respective algorithms
in the presence of additive Gaussian noise. Fig. \ref{fig2}(a)
depicts the reconstruction accuracy as a function of
signal-to-noise ratio (SNR), where we set $K=3$ and $M=25$ in our
experiments. The reconstruction accuracy is measured by the
reconstruction signal-to-noise ratio (RSNR) which is defined as
\begin{align}
\text{RSNR}=20\log_{10}\left(\frac{\|\boldsymbol{x}\|_2}{\|\boldsymbol{x}-
\boldsymbol{\hat{x}}\|_2}\right) \nonumber
\end{align}
The reconstruction accuracy as a function of the number of
measurements $M$ is plotted in Fig. \ref{fig2}(b), where we set
$K=3$ and $\text{SNR}=40\text{dB}$. It can be observed that the
proposed method achieves performance similar to the ANM and is
more robust against noise than the MC-PWGD method.


\section{Conclusions}
The line spectral estimation was studied in this paper. We
proposed a new method which formulates the observed data into a
structured tensor and casts the line spectral estimation problem
as a CP decomposition of incomplete tensors. The underlying
frequency components can be easily identified from the estimated
factor matrices. Simulation results showed that the proposed
method provides competitive recovery performance compared with
existing state-of-the-art algorithms.

\bibliography{newbib}
\bibliographystyle{IEEEtran}

\end{document}